\documentclass[11pt]{amsart}
\usepackage{geometry}                
\usepackage{graphicx}
\usepackage{amssymb}
\usepackage{epstopdf}
\newtheorem{theorem}{Theorem}
\newtheorem{lemma}[theorem]{Lemma}
\newtheorem{prop}[theorem]{Proposition}
\newtheorem*{prop*}{Proposition}

\newtheorem*{cor*}{Corollary}
\newtheorem{corollary}[theorem]{Corollary}

\theoremstyle{definition}
\newtheorem{definition}[theorem]{Definition}
\newtheorem{example}[theorem]{Example}

\newtheorem{remark}[theorem]{Remark}
\newtheorem*{remark*}{Remark}
\newtheorem*{example*}{Example}

\newcommand{\rk}{{\rm {rk}}}
\newcommand{\Z}{{\Bbb {Z}}}
\newcommand{\C}{{\Bbb {C}}}
\newcommand{\secat}{{\sf {secat}}}
\newcommand{\e}{{\mathfrak e}}
\newcommand{\CP}{{\Bbb {CP}}}
\newcommand{\h}{{\mathfrak h}}
\newcommand{\w}{{\rm w}}
\newcommand{\id}{{\rm id}}
\newcommand{\TC}{{\sf TC}}

\title{Parametrized topological complexity of sphere bundles}
\author{M. Farber}
\thanks{M. Farber was partially supported by a grant from the EPSRC}
\address{School of Mathematical Sciences, Queen Mary University of London}
\author{S. Weinberger}
\thanks{S. Weinberger was partially supported by a grant from the NSF}
\address{Department of Mathematics, University of Chicago}
\date{\today}   

\begin{document}

\maketitle

\begin{abstract} Parametrized motion planning algorithms \cite{CFW} 
 have high degree of flexibility and universality, they can work under a variety of external conditions, 
 which are viewed as parameters and form part of the input of the algorithm.
 In this paper we analyse the parameterized motion planning problem in the case of sphere bundles.
Our main results provide upper and lower bounds for the parametrized topological complexity; the upper bounds typically involve sectional categories of the associated fibrations and the lower bounds are given in terms of characteristic classes and their properties. We explicitly compute the parametrized topological complexity in many examples and show that it may assume arbitrarily large values. 
 \end{abstract}

 \section{Introduction}
 
 The motion planning problem of robotics is one of the central themes which makes possible autonomous 
 robot motion, see \cite{LV}. A motion planning algorithm takes as input the initial and the desired states of the system and produces as output a motion of the system starting at the initial and ending at the desired states. 
 A robot is \lq\lq told\rq\rq\ where it needs to go and the execution of this task, including selection of a specific route of motion, is made by the robot itself. 
 In this approach it is understood that the external conditions (such as the positions of the obstacles and the geometry of the enclosing domain) are known. 
 
 In recent papers \cite{CFW}, \cite{CFW2}, motion planning algorithms of a new type were considered. 
 These are {\it parametrized motion planning algorithms}, which, besides the initial and desired states, take as input the parameters characterising the external conditions. The output of a parametrized motion planning algorithm is a continuous motion of the system from the initial to the desired state respecting the given external conditions. The papers \cite{CFW, CFW2} laid out the new formalism and analysed in full detail the problem of moving a number $n$ of robots in the domain with $m$ unknown obstacles. The authors used techniques of algebraic topology and were able to find the answer by using a combination of upper and lower bounds. The lower bounds use the structure of the cohomology algebras. 
A brief introduction into the concept of parametrized topological complexity is given below in \S \ref{sec3}. 

The purpose of this article is to analyse the parametrized topological complexity of sphere bundles. 
The Stiefel - Whitney and Euler characteristic classes play an important role in these estimates. 
Our main results give upper and lower bounds for the parametrized topological complexity and we 
compute the parametrized topological complexity of a number of examples. 

It would be interesting to adopt the theory of weights of cohomology classes of E. Fadell and S. Husseini  \cite{FH} with the purpose of strengthening the cohomological lower bounds in application to the parametrized topological complexity. 

The authors thank the referee for very helpful comments. 
 
\section{Sectional category of sphere bundle}\label{s1}

In this section we recall some well-known results, see \cite{S}, which will be useful later in this paper. 

Let $\xi: E\to B$ be a rank $q$ vector bundle. We shall denote
$q=\rk (\xi)$ and shall write $E(\xi)$ instead of $E$ when dealing with several bundles at once. In this article we shall always assume that vector bundles are equipped with metric structures, i.e. with continuous scalar product on fibres. 

We shall denote by $\dot \xi : \dot E\to B$ the associated bundle of $(q-1)$-dimensional spheres, i.e. $\dot \xi=\xi|_{\dot E}$. Here $\dot E\subset E$ is the set of vectors of length 1. 
If $\xi$ is oriented, its Euler class $\e(\xi)\in H^q(B)$ is defined, see \cite{MS}. Here the cohomology is taken with integral coefficients. 
We shall adopt the convention of skipping $\Z$ from the notations while indicating explicitly all other coefficient groups. 

For a cohomology class $\alpha\not=0$ we shall denote by $\h(\alpha)$ its {\it height}, i.e. the largest integer $k$ such that its $k$-th power is nonzero,  $\alpha^k\not=0$. We shall also set $\h(\alpha)=0$ for $\alpha=0$. 

Recall that the {\it sectional category} (or {\it Schwarz genus} \cite{S}) of a fibration with base $B$ is defined as the minimal integer $k\ge 0$ such that 
there exists an open cover $$B=U_0\cup U_1\cup \dots \cup U_k$$ with the property that over each set $U_i$ the fibration admits a continuous section.

\begin{lemma}\label{lm1} Let $\xi: E\to B$ be a vector bundle, $q=\rk(\xi)$.  Then
\begin{enumerate}
\item[{(A)}] The sectional category of the sphere fibration $\dot \xi: \dot E\to B$ satisfies 
$$\secat(\dot \xi) \ge \h(\w_q(\xi))$$ where $\w_q(\xi)\in H^q(B;\Z_2)$ is the top Stiefel-Whitney class of $\xi$. 
\item[{(B)}] If the bundle $\xi$ is orientable then $$\secat(\dot \xi) \ge \h(\e(\xi)).$$
\item[{(C)}] Moreover, if $\xi$ is orientable and the base $B$ is a CW-complex whose dimension satisfies
$$\dim B \le q \cdot \h(\e(\xi))+q,$$
then 
\begin{eqnarray*}
\secat (\dot \xi) = \h(\e(\xi)). 
\end{eqnarray*}
\end{enumerate}
\end{lemma}
\begin{proof} (A) First we observe that $\dot\xi ^\ast(\w_{q}(\xi)) =0\in H^q(\dot E; \Z_2)$. 
Indeed, using functoriality of the Stiefel-Whitney classes, we see that $\dot\xi^\ast(\w_q(\xi))$ is the top Stiefel - Whitney class of the induced fibration $\dot \xi^\ast(\xi)$ over $\dot E(\xi)$. However this fibration admits a nonzero section $s(x)=x$ where $x\in\dot E(\xi)$. 
The top Stiefel - Whitney class of a vector bundle having a section vanishes, hence $\xi^\ast(\w_q(\xi))=
\w_q(\dot\xi^\ast(\xi))=0$. 

Finally we apply the general cohomological lower bound for the sectional category, see \cite{S}, Theorem 4; this gives $\secat(\dot \xi) \ge \h(\w_q(\xi))$. 

(B) follows similarly. 

To prove (C) we apply Theorem 3 of Schwarz \cite{S} which identifies the sectional category of $\dot \xi$ with the smallest number 
$k$ such that the $(k+1)$-fold fiberwise join $(\dot\xi )^{\ast(k+1)}$ admits a continuous section. 
Note that 
$(\dot\xi )^{\ast(k+1)}$ is fiberwise homeomorphic to the unit sphere bundle of the vector bundle $(k+1)\xi=\xi\oplus\xi\oplus\dots\oplus \xi$, the Whitney sum of $k+1$ copies of $\xi$. The obstructions for a section of $(\dot\xi) ^{\ast(k+1)}$ lie in the groups 
$$H^i(B; \pi_{i-1}(S^{q(k+1)-1})), \quad i=1, 2, \dots.$$
The first obstruction (with $i=q(k+1)$) equals $$\e(\xi)^{k+1} \, =\, \e((k+1)\xi)\, \, \in H^{q(k+1)}(B).$$ 
Taking $k=\h(\e(\xi))$ we obtain $\e(\xi)^{k+1}=0$, i.e. the first obstruction vanishes. 
The further obstructions (with $i=q(k+1)+j$ where $j=1, \, 2, \dots $) also vanish because of our assumption 
$\dim B \le q \cdot (\h(\e(\xi))+1)$. This completes the proof. 
\end{proof}
\begin{example}\label{ex2}
Let $\eta: E\to \CP^n$ denote the canonical complex line bundle over $\CP^n$. We shall view $\eta$ as a rank 2 real vector bundle. Its Euler class $\e(\eta)$ is the generator of $H^2(\CP^n)$ and $\h(\e(\eta))=n$. 
Since $\dim \CP^n =2n \le 2\cdot(n+1)$, Lemma \ref{lm1} (C) applies and gives $\secat (\eta)=n$. 
\end{example} 
\begin{example}
Let $\eta: E\to \CP^n$ be as in the previous Example. 
For $k\le n$ consider $\xi_k=k\eta= \eta\oplus\eta\oplus\dots\oplus \eta$, the Whitney sum of $k$ copies of $\eta$. 
We have $\rk(\xi_k) = 2k$, and $\e(\xi_k)=\e(\eta)^k$ implying $\h(\e(\xi_k))= \lfloor n/k\rfloor$. The inequality 
$\dim B=2n \le 2k\cdot (\lfloor n/k\rfloor +1)$ is satisfied and using Lemma \ref{lm1} we obtain 
\begin{eqnarray}\label{floor}
\secat (\xi_k) \, = \lfloor n/k\rfloor
\end{eqnarray} for any
$k=1, 2, \dots, n$. Formula (\ref{floor}) is also true for $k>n$ as then the bundle $\xi_k$ admits a section and hence $\secat(\xi_k)=0$. 
\end{example}

\begin{remark} There is a version of statement (C) of Lemma \ref{lm1} for non-orientable bundles; in this case the Euler class lies in the cohomology $e(\xi)\in H^q(B; \tilde \Z)$ with twisted coefficients and its powers $e(\xi)^k$ lie in the
groups $H^{kq}(B, (\tilde \Z)^{\otimes k})$. It is easy to see that $(\tilde \Z)^{\otimes k}=\Z $ for $k$ even and 
$(\tilde \Z)^{\otimes k}=\tilde \Z$ for $k$ odd. 
\end{remark}

\section{Parametrized topological complexity}\label{sec3}

In this section we briefly recall the notion of parametrized topological complexity which was recently introduced in \cite{CFW},  \cite{CFW2}. It is a generalization of the concept of topological complexity of robot motion planning problem introduced in \cite{F}; see also \cite{F2}. 

Let $X$ be a path-connected topological space viewed as the space of states of a mechanical system. The motion planning problem of robotics asks for an algorithm which takes as input an initial state and a desired state of the system, and produces as output a continuous motion of the system from the initial state to the desired state, see \cite{LV}. 
That is, given $(x_0, x_1)\in X\times X$, the algorithm will produce a continuous path $\gamma: I\to X$ with 
$\gamma(0)=x_0$ and $\gamma(1)=x_1$, where $I=[0,1]$ denotes the unit interval.

Let $X^I$ denote the space of all continuous paths in $X$, equipped with the compact-open topology. 
The map $\pi: X^I \to X\times X$, where $\pi(\gamma) =(\gamma(0), \gamma(1))$,  is a fibration in the sense of Hurewicz.  A solution of the motion planning problem, {\it a motion planning algorithm}, is a section of this fibration, i.e. a map $s: X\times X \to X^I$ with $\pi\circ s =\id_{X\times X}$. 
If $X$ is not contractible, no section can be continuous, see \cite{F}.
{\it The topological complexity} of $X$ is defined to be the sectional category, or Schwarz genus, of the fibration 
$\pi: X^I \to X\times X$; notation: $\TC(X) =\secat(\pi)$. In other words, $\TC(X)$ is the smallest integer $k$ for 
 which there exists an open cover $X\times X = U_0\cup U_1\cup \dots\cup U_k$ such that 
  the fibration $\pi$ admits a continuous section $s_j : U_j \to X^I $ for each $j=0, 1, \dots, k$. 
 
In the parametrized setting developed in \cite{CFW}, one assumes that the motion of the system is constrained 
by external conditions, such as obstacles or variable geometry of the containing domain. 
 The initial and terminal states of the system, as well as the motion between them, must live under the same external conditions.  
 
This situation is modelled by a fibration $p: E \to  B$, with path-connected fibers, where the base $B$ is a topological space encoding the variety of external conditions. For $b \in B$, the fiber 
 $X_b = p^{-1}(b)$ is viewed as the space of achievable configurations of the system given the constraints imposed by $b$. {\it A parametrized motion planning algorithm} takes as input initial and terminal states of the system (consistent with external conditions $b$), and produces a continuous  path between them, achievable under external conditions $b$. 
 The initial and terminal points, as well as the path between them, all lie within the same fiber $X_b$. 
 
 To define {\it the parametrized topological complexity} of the fibration $p: E \to B$ one needs to introduce  
 the associated fibration $\Pi: E^I_B \to E \times_B E,$ where $E \times_B E$ is the space of all pairs of configurations lying in the same fiber of $p$, while $E_B^I$ stands for the space of continuous paths in $E$ lying in a single fiber of $p$; the map $\Pi$ sends a path to its endpoints.
 \begin{definition}\label{def1}
 The parametrized topological complexity  
 of the fibration $p: E \to B$ is defined as the sectional category of the fibration 
 \begin{eqnarray}\label{fib}
 \Pi: E_B^I \to E \times_B E, \quad \Pi(\gamma)=(\gamma(0), \gamma(1)).\end{eqnarray}
In more detail, $$\TC[p:E\to B]:=\secat(\Pi:E_B^I \to E\times_B E)$$ is the minimal integer $k$ such that 
$E \times_B E$ admits an open cover $E \times_B E=U_0\cup U_1\cup \dots\cup U_k$ with the property that each set $U_i$ admits a continuous section of $\Pi$, where $i=0, 1, \dots, k$.
\end{definition}

Note that $\Pi: E_B^I \to E \times_B E$ is a Hurewicz fibration assuming that $p: E\to B$ is a Hurewicz fibration, see \cite{CFW2}, Proposition 2.1. 

If $B'\subset B$ is a subset and $E'=p^{-1}(B')\subset E$, then the topological complexity $\TC[p':E'\to B']$  of the restricted fibration (where $p'=p|_{E'}$) clearly satisfies $$\TC[p':E'\to B'] \le \TC[p:E\to B].$$ In particular, we obtain the inequality
\begin{eqnarray}\label{fibre}
\TC(X) \le \TC[p:E\to B]  
\end{eqnarray}
where $X$ is the fibre of $p:E\to B$. 

\begin{lemma}\label{lmtriv}
Let $p: E\to B$ be a locally trivial fibration with fibre $X$. (A) If $\TC[p:E\to B]=0$ then  $X$ is contractible. 
(B) Conversely, if the the fibre $X$ is contractible and the base $B$ is paracompact then there exists a globally defined continuous parametrized motion planning algorithm $s: E\times_BE \to E^I_B$ and therefore $\TC[p:E\to B]=0$. 
\end{lemma}
\begin{proof}
If $\TC[p:E\to B]=0$ then $\TC(X)=0$ because of (\ref{fibre}). By Theorem 1 from \cite{F} this implies that $X$ is contractible. To prove (B) we shall apply Corollary 3.2 from the paper of A. Dold \cite{D}. It implies that a locally trivial fibre bundle $p: E\to B$ with paracompact base and contractible fibre is {\it shrinkable}; this means that there exists a continuous section $\sigma: B \to E$ and a homotopy $H: E\times I\to E$ such that for any $e\in E$ one has
$H(e, 0)=e$, $H(e, 1)=\sigma p(e)$ and $p(H(e,t))=p(e)$. We may define the section 
$s: E\times_BE \to E^I_B$ by the formula
\begin{eqnarray}\label{s}
s(e, e')(t) = \left\{
\begin{array}{lll}
H(e, 2t), &\mbox{for}& 0\le t\le 1/2,\\ \\
H(e', 2-2t), &\mbox{for}& 1/2\le t\le 1,
\end{array}
\right.
\end{eqnarray}
where $(e, e')\in E\times_BE$ and $t\in I$. Since $H(e, 1)=\sigma p(e) = \sigma p(e') = H(e', 1)$, we see that 
both parts of the formula (\ref{s}) match and hence $s$ is continuous. We clearly have $s(e, e')(0)=e$ and $s(e, e')(1)=e'$. Besides, $p(s(e, e')(t)) = p(e) = p(e')$, i.e. $s$ is a continuous 
parametrized motion planning algorithm. 
\end{proof}

Next we mention the upper and lower bounds for the parametrized topological complexity established in \cite{CFW}. 

\begin{prop} [Proposition 7.2 from \cite{CFW}]\label{prop1} Let $p: E \rightarrow  B$ be a locally trivial fibration with fiber $X$, where the spaces 
$E, B, X$ are CW-complexes. Assume that the fiber $X$ is r-connected. Then
\begin{eqnarray}
\TC[p:E\to B] <  \frac{2\dim X+\dim B+1}{r+1}.\end{eqnarray}
\end{prop}

\begin{prop}[Proposition 7.3 from \cite{CFW}]\label{prop2}  Let $p: E \rightarrow  B$ be a fibration with path-connected fiber. Consider the diagonal map $\Delta : E \rightarrow  E \times_B E$, where $\Delta (e) = (e,e)$. Then the parametrized topological complexity $\TC[p: E \rightarrow  B]$ is greater than or equal to the 
cup-length of the kernel $\ker[\Delta^\ast  : H^\ast (E \times_B E; R) \rightarrow  H^\ast (E; R)]$, where $R$ is an arbitrary coefficient ring. 
\end{prop}

\begin{prop}\label{prop3}
[Proposition 4.7 from \cite{CFW}] If $p: E \rightarrow  B$ is a locally trivial fibration, and the spaces $E $ and $B$ are metrizable separable ANRs, then in Definition \ref{def1}, instead of open covers one may use arbitrary covers of $E \times_B E$ or, equivalently, arbitrary partitions
$$E\times_BE=F_0\sqcup F_1\sqcup \cdot \cdot \cdot \sqcup F_k, \quad F_i\cap F_j=\emptyset ,\quad i\not=j$$
admitting continuous sections $s_i : F_i \rightarrow  E_B^I$, where $i = 0, 1, . . . , k.$
\end{prop}

\begin{example}\label{ex10} As an illustration consider the canonical complex line bundle $\eta: E\to \CP^n=B$ viewed as a real rank 2 vector bundle. 
The unit sphere bundle $\dot\eta: \dot E(\eta)\to \CP^n$ is a principal $S^1$-bundle, its total space $\dot E(\eta)$ is the sphere $S^{2n-1}$, the set of unit vectors $z\in \C^n$. 
The unit circle $S^1\subset \C$ acts by multiplication, this action is free and the quotient is $\CP^n$. We claim that 
\begin{eqnarray}\label{hopf}
\TC[\dot\eta: \dot E(\eta)\to \CP^n]=1.\end{eqnarray} 
Indeed, using (\ref{fibre}) we get 
$\TC[\dot\eta: \dot E(\eta)\to \CP^n]\ge \TC(S^1)=1$. To obtain the inverse inequality we consider the following partition 
$$\dot E(\eta)\times_B\dot E(\eta)= F_0\sqcup F_1$$
where $F_0$ is the set of all pairs $(z_1, z_2)$ of unit vectors $z_1, z_2\in S^{2n-1}$ lying in the same fibre but not antipodal, i.e. $z_1\not= -z_2$; the set $F_1$ is the set of antipodal pairs $(z, -z)$. 
For $(z_1, z_2)\in F_0$ we can write $z_2/z_1 = e^{i\phi}$ where $\phi\in (-\pi, \pi)$ and a continuous section $s_0$ of the fibration (\ref{fib}) over $F_0$ can be defined as follows: 
\begin{eqnarray*}
s_0(z_1, z_2)(t) \, =\, e^{i\phi t}z_1, \quad t\in [0,1].
\end{eqnarray*}
On the other hand, over $F_1$ we can define a continuous section $s_1$ where 
\begin{eqnarray}\label{rotate}
s_1(z, -z)(t) \, =\, e^{i\pi t}\cdot z, \quad t\in [0,1].
\end{eqnarray}
This proves (\ref{hopf}).

\end{example}

This example is a special case of a more general statement that the parametrized topological complexity of any principal bundle equals the Lusternik - Schnirelmann category of the fibre, see Proposition 4.3 in \cite{CFW}. 

The main result of \cite{CFW} is the computation of the parametrized topological complexity of the Fadell - Neuwirth fibration which, in term of robotics, can be understood as the complexity of controlling multiple robots in the presence of multiple movable obstacles. 

\section{The cup-length associated with a section}

Material of this section will play an auxiliary role in the sequel. We shall use notations introduced in the beginning of \S \ref{s1}. 


Let $\xi: E\to B$ be an oriented vector bundle of rank $q\ge 2$ equipped with scalar product structure $\langle \, , \, \rangle$.  
As above, let $\dot\xi: \dot E\to B$  denote the unit sphere bundle; its fibre is homeomorphic to $S^{q-1}$.  
In this section we shall assume that the fibration $\dot \xi: \dot E\to B$ has a continuous section $s: B\to \dot E$ and our goal will be to identify the kernel 
\begin{equation}\label{kernel}
\ker [s^\ast: H^\ast(\dot E)\to H^\ast(B)]
\end{equation} 
and its {\it cup-length}, i.e. the length of the longest nontrivial products of elements of this kernel. This result will be used in the following sections to estimate the parametrized topological complexity from below. 

We shall use the following remark. The oriented sphere fibration $\dot \xi: \dot E\to B$ admits a cohomological extension of the fibre (see \cite{Sp}, chapter 5, \S 7) if and only if its Euler class vanishes, $\e(\xi)=0$. In particular, any oriented sphere fibration with trivial Euler class $\e(\xi)=0$ satisfies the conclusion of the Leray - Hirsch theorem, see \cite{Sp}.

Let $\dot F\subset \dot E$ denote the set $\{e\in \dot E; \, e\perp s(\xi(e))\}$; it is the set of unit vectors perpendicular to the section. The projection 
$$\eta: \, \dot F \to B, \quad \eta=\dot\xi|_{\dot F}$$
is an oriented bundle of $(q-2)$-dimensional spheres. Let $\e(\eta) \, \in\,  H^{q-1}(B)$ denote the Euler class of $\eta$. The mod-2 reduction of the class $\e(\eta)$ equals the Stiefel - Whitney class $\w_{q-1}(\xi)\in H^{q-1}(B; \Z_2)$ of $\xi$. 

\begin{theorem} \label{section}
The cup-length of the kernel (\ref{kernel}) equals $\h(\e(\eta))+1$ where $\h(\e(\eta))$ denotes the height of the Euler class
$\e(\eta)\in H^{q-1}(B)$.   
\end{theorem}

\begin{proof} Let $U\in H^{q-1}(\dot E)$ denote a fundamental class: for any $b\in B$ the restriction $U|_{\dot E_b}$ is the fundamental class of the fibre $\dot E_b$. By the Leray - Hirsch theorem every cohomology class in $H^\ast(\dot E)$ has a unique representation in the form $$\xi^\ast(u)+\xi^\ast(v)\smile U,$$ where $u, v\in H^\ast(B)$. 

Let $W, W'\subset\dot E$ denote the following subsets: 
$$W=\{e\in \dot E; \langle e, s(\xi(e))\rangle \ge 0\}\quad \mbox{and}\quad
W'=\{e\in \dot E; \langle e, s(\xi(e))\rangle \le 0\}.$$ Clearly $W\cup W'=\dot E$ and $W\cap W'=\dot F$. 
One can identify $W$ with the unit disc bundle of the sphere bundle $\dot F$. 
Therefore the quotient $\dot E/W'= W/\dot F$ can be naturally identified with the Thom space of the fibration 
$\eta$. 

Next we observe that the fundamental class $U\in H^{q-1}(\dot E)$ can be chosen such that $U|_{W'}=0$. 
Indeed, starting with an arbitrary choice  $U'$ we can replace it by 
$U=U'-\xi^\ast(x)$ where $x\in H^{q-1}(B)$ is such that $(\xi|_{W'})^\ast(x) = U'|_{W'}$. Here we use the observation that $\xi|_{W'}: W' \to B$ is a homotopy equivalence and hence the class $x$ mentioned above exists and is unique. With this choice clearly $U|_{W'}=0$.

Once the fundamental class $U$ satisfies $U|_{W'}=0$ we have the following formulae which 
fully describe the multiplicative structure of $H^\ast(\dot E)$:
\begin{eqnarray}\label{two}
s^\ast(U)= \e(\eta) \, \in\,  H^{q-1}(B)
\end{eqnarray}
and
\begin{eqnarray}\label{one}
U\smile U = \xi^\ast(\e(\eta)))\smile U\, \in \, H^{2(q-1)}(\dot E).
\end{eqnarray} 
To prove (\ref{one}) 
we note that 
$U|_W=0$ implies that the class $U$ can be refined to a relative class $\tilde U\in H^{q-1}(\dot E, W')
= H^{q-1}(\dot E/W')=H^{q-1}(W/\dot F)$. We already mentioned that the quotient 
$\dot E/W'$ can be identified with the Thom space of the vector bundle having $\eta$ as its unit sphere bundle. Examining the long exact sequence in cohomology
$$\dots\to H^{q-2}(\dot E) \stackrel{\simeq}\to H^{q-2}(W')\to H^{q-1}(\dot E, W')\to H^{q-1}(\dot E)$$
 we see that the refinement $\tilde U$ is unique and coincides with the Thom class. 
 Now, by the definition (see \cite{MS}, \S 9), we have 
 \begin{eqnarray}\label{11}
  \e(\eta) \, =\,  s^\ast(\tilde U|_{W})\, =\,  s^\ast(U),
  \end{eqnarray} 
 which proves (\ref{two}). 
From (\ref{11}) we also obtain
$$\tilde U\smile \tilde U= (\tilde U|_W)\smile \tilde U = \xi^\ast(\e(\eta))\smile \tilde U\, \, \in \, \, H^{2(q-1)}(\dot E, W').$$
Applying 
the restriction homomorphism $H^{2(q-1)}(\dot E, W') \to H^{2(q-1)}(\dot E)$
to both sides of this equality 
gives (\ref{one}). 
%

Note that the order of the factors in the RHS of formula (\ref{one}) is irrelevant: if $q$ is odd then the classes commute and for $q$ is even the Euler class
$\e(\eta)$ has order two. 

Consider now an arbitrary class $x\in H^\ast(\dot E)$ satisfying $s^\ast(x)=0$. We can write 
$$x=\xi^\ast(\alpha) +\xi^\ast(\beta)\smile U \quad \mbox{with}\quad \alpha, \,  \beta\in H^\ast(B).$$
Applying $s^\ast$ and using (\ref{two}) we get
\begin{eqnarray}\label{three}
s^\ast(x)\, =\, 0\, =\, \alpha +\beta\smile \e(\eta). 
\end{eqnarray}
Conversely, any two classes $\alpha, \, \beta$ satisfying (\ref{three}) produce a class $x=\xi^\ast(\alpha) +\xi^\ast(\beta)\smile U$ lying in the kernel of $s^\ast$. 
A particular choice $\alpha=-\e(\eta)$ and $\beta=1$ gives the class $$x_0= U-\xi^\ast(\e(\eta))\, \in H^{q-1}(\dot E).$$
Using (\ref{one}) we have
$$x_0^2=U^2 - 2\xi^\ast(\e(\eta))\smile U + \xi^\ast(\e(\eta))^2 = - \xi^\ast(\e(\eta))\smile x_0$$
and we obtain by induction 
\begin{eqnarray}
x_0^n &=& (-1)^{n-1} \xi^\ast(\e(\eta)^{n-1})\smile x_0\\
\nonumber
&=& (-1)^{n}\xi^\ast(\e(\eta)^{n}) + (-1)^{n-1} \xi^\ast(\e(\eta))^{n-1}\smile U. 
\end{eqnarray}
For $n=h(\e(\eta))+1$ the class $x_0^n$ equals $(-1)^{n-1} \xi^\ast(\e(\eta))^{n-1}\smile U$ and is
obviously nonzero (as follows from the Leray - Hirsch theorem). This implies that the cup-length of the kernel $\ker s^\ast$ is at least $\h(\e(\eta))+1$. 

If $x=\xi^\ast(\alpha)+\xi^\ast(\beta)\smile U\, \in \, H^\ast(\dot E)$ is an arbitrary class with $s^\ast(x)=0$ then
$\alpha =-\beta\smile \e(\eta)$ (see above) and thus $x= \xi^\ast(\beta)\smile x_0$. In other words, the 
kernel $\ker s^\ast\subset H^\ast(\dot E)$ is the principal ideal generated by the class $x_0$. We see that the cup-length of the kernel 
equals the highest nonzero power of $x_0$ which, as we have shown above, is $\h(\e(\eta))+1$. 
\end{proof}

The following Corollary is an analogue of Theorem \ref{section} where we use $\Bbb Z_2$ coefficients. The role of the Euler class plays the top Stiefel - Whitney class of the bundle $\eta$ of vectors orthogonal to the section. The advantage of this statement is that the answer is given in terms of the original bundle $\xi$ and its characteristic class $\w_{q-1}(\xi)\in H^{q-1}(B;\Bbb Z_2)$. 

\begin{corollary}\label{cor11}
Let $\xi: E\to B$ be a rank $q\ge 2$ vector bundle (not necessarily orientable). Let $s: B \to \dot E$ be a continuous section of the unit sphere bundle. Then the cup-length of the kernel $\ker[s^\ast: H^\ast(\dot E; \Z_2) \to H^\ast(B;\Z_2)]$ equals 
$\h(\w_{q-1}(\xi))+1$. 
\end{corollary}
\begin{proof}
One repeats the arguments of the proof of Theorem \ref{section} replacing the integer coefficients by $\Z_2$. 
The bundle $\eta: \, F \to B, \quad \eta=\xi |_{F}$ is the bundle of vectors orthogonal to the section, i.e. 
$F= \{e\in E; \, e\perp s(\xi(e))\}$ and the arguments of the proof of Theorem \ref{section} show that 
the the kernel $\ker[s^\ast: H^\ast(\dot E; \Z_2) \to H^\ast(B;\Z_2)]$ is the principal ideal generated by the class 
$x_0=U-\xi^\ast(\w_{q-1}(\eta))$. The height of $x_0$
equals one plus the height of the class 
$\w_{q-1}(\eta)$. 
However, $\xi=\eta\oplus \epsilon$ where $\epsilon$ is the trivial line bundle 
determined by the section and therefore $\w_{q-1}(\eta) =\w_{q-1}(\xi)$ and the result follows.
\end{proof}

\section{Parametrized topological complexity of sphere bundles}

Let $\xi: E\to B$ be an oriented rank $q\ge 2$ vector bundle equipped with fibrewise scalar product. 
Let $\dot\xi: \dot E \to B$ denote the unit sphere bundle; its fibre is the sphere of dimension $q-1$. 
Our goal is to estimate the parametrized topological complexity ${\sf TC}[\dot\xi: \dot E\to B]$. 
By Proposition \ref{prop1} we have an upper bound
\begin{eqnarray}\label{upper12}
{\sf TC}[\dot\xi: \dot E\to B] < 2 +\frac{\dim B+1}{q-1}.
\end{eqnarray}

To state our result, consider the bundle 
\begin{eqnarray}\label{ddot}
\ddot\xi: \ddot E\to \dot E,
\end{eqnarray} 
where $\ddot E\subset \dot E\times_B \dot E$ is the space of pairs of mutually orthogonal unit vectors $(x, y)\in \dot E\times_B \dot E$, \, $x\perp y$. The projection
$\ddot \xi$ acts as $\ddot \xi(x, y)=x.$ The map (\ref{ddot}) is an oriented locally trivial fibration with fibre sphere of dimension $q-2$. Consider its Euler class
\begin{eqnarray}\label{eu}
\e(\ddot \xi) \in H^{q-1}(\dot E). 
\end{eqnarray}
An obvious property of the class $\e(\ddot \xi)$ is that for any point $b\in B$ the restriction $\e(\ddot\xi)|_{\dot E_b}$ is the Euler class of the tangent bundle of the sphere $\dot E_b$. 

\begin{remark}\label{rk13}
A section $s$ of bundle (\ref{ddot}) associates with a unit vector $e\in \dot E(\xi)$ a unit vector $s(e)$ which is perpendicular to $e$ and the integer $\secat (\ddot \xi)$ is a measure of complexity of construction such a section $s$ globally, i.e. over all $\dot E$. In particular, $\secat (\ddot \xi)=0$ if the vector bundle $\xi$ admits a complex structure: in this case one can define the section by $s(e) = \sqrt{-1}\cdot e$. 
\end{remark}


\begin{theorem} \label{thm1} 
 One has the estimates
\begin{eqnarray}\label{estim} \label{15}
\h(\e(\ddot\xi))+1 \, \le \, {\sf TC}[\dot\xi: \dot E\to B]\le 
\secat(\ddot \xi)+1.\end{eqnarray} 
Moreover, if $B$ is a CW-complex satisfying $\dim B\le (q-1) \cdot  \h(\e(\ddot\xi))$ then 
\begin{eqnarray}\label{eq2}\label{16}
{\sf TC}[\dot\xi: \dot E\to B]\, =\,  \h(\e(\ddot\xi))+1 \, =\,  \secat(\ddot \xi) +1.
\end{eqnarray}
\end{theorem}
\begin{proof}
Consider the diagonal map $\Delta: \dot E \to \dot E\times_B \dot E$ and apply Proposition \ref{prop2}; we obtain that the parametrized topological complexity ${\sf TC}[\dot\xi: \dot E\to B]$ is greater than or equal to the cup-length of the kernel
$\ker [\Delta^\ast: H^\ast(\dot E\times_B \dot E) \to H^\ast(\dot E)]$. 
However, $\Delta$ is a section of the sphere fibration 
$\dot E\times_B \dot E\to \dot E$ given by projection on the first vector; hence we may apply Theorem \ref{section} which describes the cup-length of the kernel of the induced map. The bundle of vectors perpendicular to the section is exactly the bundle $\ddot\xi$. 
By Theorem \ref{section} the cup-length of the kernel $\ker [\Delta^\ast: H^\ast(\dot E\times_B \dot E) \to H^\ast(\dot E)]$ equals $\h(\e(\ddot \xi))+1$. This gives the lower bound in (\ref{15}). 

To prove the right inequality in (\ref{estim}) consider the set $U\subset \dot E\times_B\dot E$ consisting of pairs $(e, e')\in \dot E\times_B \dot E$ with $e\not= -e'$. Over $U$, we can define a continuous motion planning algorithm 
$s: U\to \dot E_B^I$
by setting
\begin{eqnarray}\label{notantipodal}
s(e, e')(t) = \frac{(1-t)e +te'}{||(1-t)e +te'||}, \quad t\in [0,1].
\end{eqnarray}

In view of Proposition \ref{prop3} it remains to construct a motion planning algorithm over the complementary set 
$V=\{(e, -e); \, e\in \dot E\}\subset \dot E\times_B\dot E.$ Denote  by
$p_1, p_2: V\to \dot E$ the projections $p_1(e, -e)=e$ and $p_2(e, -e)=-e .$

Consider again the bundle $\ddot\xi: \ddot E\to \dot E$ and suppose that $A\subset \dot E$ is a subset such that the bundle $\ddot \xi$ admits a continuous section $s_A: A \to \ddot E$ over $A$. 
Using this section we may construct a section $s'_A$ of the fibration 
$$\Pi: \dot E_B^I \to \dot E\times_B \dot E$$
 over the set $p_1^{-1}(A)$ as follows: 
\begin{eqnarray}\label{sec18}
s'_A(e, -e)(t) \, =\, \cos\left(t\pi\right)\cdot e \, + \, \sin\left(t\pi\right)\cdot s_A(e), \quad t\in [0,1].
\end{eqnarray}

Let $\dot E=A_0\cup A_1\cup \dots\cup A_k$ be an open covering, where
$k=\secat(\ddot \xi)$, with the property that $\ddot\xi: \ddot E\to \dot E$ admits a continuous section over each 
$A_i$. Then the sets $p_1^{-1}(A_i)$ cover $V$ and over each of these sets the fibration $\Pi$ admits a continuous section. Thus we get an inequality
$ {\sf TC}[\dot\xi: \dot E\to B]\le \secat(\ddot\xi)+1$.

Finally we apply Lemma \ref{lm1} which claims that the sectional category of $\ddot\xi$ equals 
$\h(\e(\ddot \xi))$ under an additional assumption that $\dim \dot E \le (q-1)\cdot \h(\e(\ddot \xi)) +(q-1)$ which is equivalent to 
$\dim B\le (q-1)\cdot \h(\e(\ddot \xi))$. Hence under this assumption we obtain
$ {\sf TC}[\dot\xi: \dot E\to B] \, =\,  \h(\e(\ddot \xi))+1\, =\, \secat(\ddot\xi)+1$.
This completes the proof. 
\end{proof}

%

\begin{corollary}\label{cor15}
For a vector bundle $\xi: E\to B$ satisfying $\secat (\ddot \xi ) =0$ one has $\TC[\dot \xi: \dot E\to B]=1.$ 
\end{corollary}
\begin{proof}
The inequality (\ref{15}) gives 
$\TC[\dot \xi: \dot E\to B]\le 1.$ On the other hand, $\TC[\dot \xi: \dot E\to B]\ge \TC(S^{2r-1})=1$ by (\ref{fibre}).
\end{proof}

\begin{example}\label{ex15}
Consider the canonical rank 2 bundle $\xi$ over $\CP^n$ as in Example \ref{ex10}. In this case 
$\dot E(\xi)=S^{2n-1}$ and 
the bundle
$\ddot\xi: \ddot E\to \dot E$ is the trivial bundle with fibre $S^0$, i.e. $$\secat (\ddot\xi)=0=\h(\e(\ddot \xi)).$$ 
We obtain from (\ref{estim}) that ${\sf TC}[\dot \xi: S^{2n-1} \to \CP^n]=1$ confirming the result of Example
\ref{ex10}. 
\end{example}

Generalising Example \ref{ex15} we may state:

\begin{corollary}  For any vector bundle $\xi: E\to B$ 
of even rank $\rk(\xi)=2r$ admitting a complex structure, one has
$$\TC[\dot \xi: \dot E\to B]=1 = \TC(S^{2r-1}).$$
\end{corollary}

\begin{proof} In this case $\secat(\ddot\xi)=0$ (by Remark \ref{rk13}) and the result follows from Corollary \ref{cor15}.
\end{proof}


\begin{remark}
Introducing the bundle $\ddot\xi$ over $\dot E$ and using its sectional category to estimate the parametrized topological complexity we made an approximation of the space of paths on the sphere connecting a pair of antipodal points by the sphere of one dimension below. This sphere is however only the first term in the 
James' construction $JS^{q-2}$, see \cite{J},  which gives a CW complex having the homotopy type of this space of paths. 
\end{remark}

Note that for $q$ even the Euler class $\e(\ddot\xi)\in H^{q-1}(\dot E)$ has order 2, i.e. $2\cdot \e(\ddot\xi)=0$. 
We shall focus below on the case when $q$ odd. Compared with Theorem \ref{thm1}, Corollary \ref{cor3} stated below has the advantage of 
dealing with cohomology of the base $B$.

\begin{corollary}\label{cor3} For $q\ge 3$ odd, 
let $\eta: E(\eta)\to B$ be an oriented vector bundle of rank $q-1$. 
Let 
$\xi=\eta\oplus \epsilon$ be the sum where $\epsilon$ is the trivial line bundle over $B$. 
 Then one has
\begin{eqnarray}\label{plus2}
{\sf TC}[\dot \xi: \dot E(\xi)\to B] \ge \h(\e(\eta))+1.
\end{eqnarray}
Moreover, if the height $\h(\e(\eta))$ is even and the integral cohomology of the base $B$ in dimension 
$(q-1)\cdot \h(\e(\eta))$ has no 2-torsion 
 then 
\begin{eqnarray}\label{plus2}
{\sf TC}[\dot \xi: \dot E(\xi)\to B] \ge \h(\e(\eta))+2;
\end{eqnarray}
\end{corollary}
%

%

\begin{proof} Consider the Euler class $\e(\ddot \xi)\in H^{q-1}(\dot E(\xi))$. Applying the Leray - Hirsch theorem we see that any class in $H^{q-1}(\dot E(\xi))$ has a unique representation as $\dot\xi^\ast(a)+bU$
where $U\in H^{q-1}(\dot E)$ is a fundamental class, $a\in H^{q-1}(B)$ and $b\in \Z$. Let $s: B\to \dot E(\xi)$ be the section determined by the trivial summand $\epsilon$. We showed in the proof of Theorem \ref{section} that the fundamental class $U$ can be chosen so that 
\begin{eqnarray}\label{two2}
s^\ast(U)=\e(\eta),
\end{eqnarray} see formula (\ref{two}). 
Note that 
\begin{eqnarray}\label{eta}
s^\ast(\ddot\xi) =\eta.
\end{eqnarray}
Besides, 
\begin{eqnarray}\label{thus}
\e(\ddot\xi) =\dot\xi^\ast(a) + 2U, \quad \mbox{for some class}\quad a\in H^{q-1}(B).
\end{eqnarray}
Indeed, the class $\e(\ddot\xi) $
restricted to each fibre $\dot E_b(\xi)$ equals twice the fundamental class of the sphere
$\dot E_b(\xi)\simeq S^{q-1}$ (here we use our assumption that $q$ is odd, and hence the Euler characteristic of $S^{q-1}$ equals 2). 
Applying $s^\ast$ to both sides of equation
(\ref{thus}) we find $s^\ast(\e(\ddot \xi))= \e(s^\ast(\ddot\xi))= \e(\eta)$, and 
$s^\ast(\dot\xi^\ast(a)) =a$ which together with (\ref{two2}) give $a=-\e(\eta).$ 
Therefore we have
\begin{eqnarray}\label{thus2}
\e(\ddot\xi) = - \dot\xi^\ast(\e(\eta)) + 2U. 
\end{eqnarray}
Using $U^2=\dot\xi^\ast(\e(\eta))\smile U$ (see (\ref{one})) we find
$\e(\ddot\xi)^2 = \dot\xi^\ast(\e(\eta)^2)$ and therefore the even and odd powers of the class 
$\e(\ddot\xi)$ are as follows
\begin{eqnarray}\label{26}
\e(\ddot\xi)^{2n} = \dot\xi^\ast(\e(\eta)^{2n})
\end{eqnarray}
and 
\begin{eqnarray}\label{27}
\e(\ddot\xi)^{2n+1} = -\dot\xi^\ast(\e(\eta)^{2n+1}) + 2 \dot\xi^\ast(\e(\eta)^{2n})\smile U.
\end{eqnarray}
%
From formulae (\ref{26}) and (\ref{27}) we see that the height $\h(\e(\ddot\xi))$ either equals to $\h(\e(\eta))$ or it equals $ \h(\e(\eta))+1$; the second possibility happens iff $\h(\e(\eta))$ is even and the group $H^{(q-1)\h(\e(\eta))}(B)$ has no 2-torsion. 

Applying Theorem \ref{thm1} completes the proof. 
\end{proof}

\begin{example}\label{ex11} Consider the situation of Corollary \ref{cor3} in the case when $\eta: E(\eta) \to \CP^n$ is the canonical bundle over the complex projective space as in Example \ref{ex10}. 
Taking $\xi=\eta\oplus\epsilon$  we have  $\rk (\xi) = q= 3$ is odd and $\h(\e(\eta))=n$. 
By Corollary \ref{cor3} we get 
${\sf TC}[\dot\xi: \dot E(\xi)\to \CP^n]\ge n+1$ and moreover for $n$ even
${\sf TC}[\dot\xi: \dot E(\xi)\to \CP^n]\ge n+2$. 
On the other hand, the upper bound (\ref{upper12}) gives 
${\sf TC}[\dot\xi: \dot E(\xi)\to \CP^n]\le n+2$. Thus, we see that 
$${\sf TC}[\dot\xi: \dot E(\xi)\to \CP^n]= n+2$$
for all even $n$. In particular, we see that the parametrized topological complexity of sphere bundles 
can be arbitrarily large. This contrasts the situation with the usual (i.e. unparametrized) topological complexity which takes the values $1$ and $2$ only for spheres. 

 Finally we describe an explicit parametrized motion planning {\bf algorithm} having complexity $n+2$ for the unit sphere bundle 
associated with the vector bundle $\xi=\eta\oplus\epsilon$ over $B=\CP^n$ as considered in Example \ref{ex11}. 
We shall describe a partition 
\begin{eqnarray}
\dot E(\xi)\times_{B}\dot E(\xi)\, =\, F_0\sqcup F_1\sqcup\dots\sqcup F_{n+2}
\end{eqnarray}
and a continuous section $s_i$ of the fibration 
$$\Pi: \dot E(\xi)^I_{B}\, \to \, \dot E(\xi)\times_{B}\dot E(\xi)$$
over each of the sets $F_i$ where $i=0, 1, \dots, n+2$.  

The set $F_0\subset \dot E(\xi)\times_{B}\dot E(\xi)$ will be defined as the set of pairs $(x, y)\in \dot E(\xi)\times_{B}\dot E(\xi)$ with $x\not= -y$.
The section $s_0$ over $F_0$ can be defined by formula (\ref{notantipodal}). 

The unit sphere bundle of the trivial summand $\epsilon$ gives the sections $\pm \sigma: B \to \dot E(\epsilon)\subset \dot E(\xi)$. Let $\dot E(\xi)^\ast$ denote the complement $\dot E(\xi)- \dot E(\epsilon)$.
We define the set 
$F_1\subset \dot E(\xi)\times_{B}\dot E(\xi)$ to be the set of all pairs $(x, -x)$ with $x \in \dot E(\xi)^\ast$.
Let ${\rm pr}: \dot E(\xi)^\ast\to \dot E(\eta) \subset \dot E(\xi)$ denote the retraction 
given by the formula
\begin{eqnarray}\label{pr}
{\rm pr}(x) \, =\, \frac{x-\langle x, \sigma(b)\rangle \cdot \sigma(b)}{||x-\langle x, \sigma(b)\rangle \cdot \sigma(b)||} \quad \mbox{where}\quad b=\xi(x).
\end{eqnarray}
Here the symbol $\langle \, , \, \rangle$ denotes scalar product in the fibre. The deformation 
$$\alpha_t(x)  \, =\, \frac{x- t\cdot \langle x, \sigma(b)\rangle \cdot \sigma(b)}{||x-t \cdot \langle x, \sigma(b)\rangle \cdot \sigma(b)||} \quad \mbox{where}\quad b=\xi(x), \, t\in [0,1],$$
satisfies $\alpha_0(x) =x$ and $\alpha_1(x)={\rm pr}(x)$. The homotopy $t\mapsto (\alpha_t(x), \alpha_t(-x))$ deforms the initial pair $(x, -x)$ to a pair of antipodal points lying in the equatorial sphere $\dot E(\eta)\subset \dot E(\xi)$. Note that the circle $S^1$ acts freely on $\dot E(\eta)$, see Example \ref{ex10}. 
We may define a continuous section $s_1$ over $F_1$ by setting $s_1(x, -x)(t)$ to be the concatenation of the following three paths: (a) the deformation $\alpha_t(x)$, (b) the section (\ref{rotate}) of Example \ref{ex10}, and (c) the reverse of the deformation $\alpha_t(-x)$. In more detail, 
$$s_1(x, -x)(t) =\left\{
\begin{array}{lll}
\alpha_{3t}(x), & \mbox{for} & t\in [0, 1/3],\\ \\
e^{i\pi(3t-1)}\cdot {\rm pr}(x), & \mbox{for} & t\in [1/3, 2/3],\\ \\
\alpha_{3(1-t)}(-x), & \mbox{for} & t\in [2/3, 1].
\end{array}
\right.
$$

Finally we define the sections $s_i: F_i \to \dot E(\xi)^I_B$ for $i=2, 3, \dots, n+2$ as follows. The base $B=\CP^n$ has the well-known cell decomposition $\CP^n=e^0\sqcup e^2\sqcup\dots\sqcup e^{2n}$ with a single cell 
$e^{2i}$ in each even dimension $2i\le 2n$. For $i=2, 3, \dots, n+2$,
let $F_i$ denote the set of pairs $(x, -x)$ with $x=\pm \sigma(b)$ for $b=\xi(x)$ lying in the cell $e^{2(i-2)}$. 
Since the cell $e^{2(i-2)}$ is contractible, the bundle $\eta$ admits a continuous section $\phi_i$ over $e^{2(i-2)}$. Hence we may define the section $s_i: F_i \to \dot E(\xi)^I_B$ by the formula 
$$s_i(x, -x)(t) = \cos(\pi t) \cdot x + \sin(\pi t) \cdot \phi_i(\xi(x)),$$
similarly to (\ref{sec18}). Here we assume that the Euclidean structure on the vector bundle $\xi$ is the orthogonal sum of the Euclidean structures of $\eta$ and $\epsilon$. 
\end{example}

We conclude the paper with the following observations.

Below we always assume that the base $B$ is an ANR. 

\begin{lemma}\label{lm21} Let $\xi: E\to B$ be a vector bundle such that 
$\xi=\eta\oplus \tau$ and $\rk(\tau)\ge 2$. 
Then 
\begin{eqnarray}\label{eneqq}
\hskip 1cm
\secat (\ddot \xi: \ddot E(\xi)\to \dot E(\xi))\, \le \, \secat(\ddot \tau: \ddot E(\tau) \to \dot E(\tau))+  \secat(\dot \tau: \dot E(\tau)\to B) +1\end{eqnarray}
 and consequently
\begin{eqnarray}\label{eneq}
\TC[\dot\xi : \dot E\to B]\le \secat(\ddot \tau: \ddot E(\tau)\to \dot E(\tau) )+
\secat(\dot \tau: \dot E(\tau)\to B) + 2.\end{eqnarray} 
\end{lemma}
\begin{proof} It is enough to prove the inequality (\ref{eneqq}) as 
 (\ref{eneq}) follows from (\ref{eneqq}) and Theorem \ref{thm1}. 

If $\xi = \eta\oplus \tau$ then for any point of the base $b\in B$
we have $E(\xi)_b=E(\eta)_b \oplus E(\tau)_b$. 
The scalar product can be chosen so that the spaces 
$E(\eta)_b, \, E(\tau)_b\subset E(\xi)_b$ are mutually orthogonal. We shall denote by $P^\eta_b$ and 
$P^\tau_b$ the orthogonal projections of $E(\xi)_b$ onto $E(\eta)_b$ and $E(\tau)_b$ correspondingly. 


Let $k$ denote $\secat(\ddot \tau: \ddot E(\tau)\to \dot E(\tau))$ and let $\ell$ denote
$\secat(\dot \tau: \dot E(\tau)\to B)$. Let $\dot E(\tau)= G_0\sqcup G_1\sqcup \dots\sqcup G_k$ be a partition such that for each $j=0, \dots, k$ there exists a continuous map $\sigma_j: G_j\to \dot E(\tau)$ with the property that for every $e\in G_j$ the vectors $e$ and $\sigma_j(e)$ lie in the same fibre and are perpendicular to each other.  
Besides, let $B_0\sqcup B_1\sqcup \dots\sqcup B_\ell=B$ be a partition of the base $B$ with continuous sections $\nu_i: B_i\to \dot E(\tau)$, where $i=0, \dots, \ell$. 

We want to show that $\dot E(\xi)$ can be partitioned as 
$$\dot E(\xi)= F_0\sqcup F_1\sqcup \dots\sqcup F_k \sqcup F_{k+1}\sqcup \dots\sqcup F_{k+\ell+1}$$
such that there exist continuous sections 
$s_i: F_i\to E(\ddot\xi)$ 
of the fibration $\ddot \xi$. 
 In view of the result of J. M. Garcia-Calcines \cite{JGC}, this is equivalent to $\secat (\ddot \xi)\le k+\ell +1$.

For $i=0, 1, \dots, k$ we set 
$$F_i=\{e\in \dot E(\xi); P_{\xi(e)}^\tau(e)\not=0, \, \mbox{and} \, ||P_{\xi(e)}^\tau(e)||^{-1} \cdot P_{\xi(e)}^\tau(e) \in G_i\}.$$
and for $i=k+1, \dots, k+\ell +1$, we set
$$F_i= \{e\in \dot E(\xi); \, P_{\xi(e)}^\tau(e)=0 \, \, \mbox{and}\, \, \xi(e)\in B_{i-k-1}\}.$$

Next we construct the continuous sections $s_i: F_i\to E(\ddot \xi)$.

For $i=0, 1, \dots, k$, given a unit vector $e\in F_i$, 
consider $e'=P_{\xi(e)}^\tau(e)$ which is a nonzero vector of $E(\tau)_b$, where
$b=\xi(e)$. Then $e''=||e'||^{-1} \cdot e'$ is a unit vector and $\sigma_i(e'')\in \dot E(\tau)_b$ is a unit vector 
satisfying $\sigma_i(e'')\perp e'$. In particular we see that the unit vectors $e, \, \sigma_i(e'')\in E(\xi)_b$ are linearly independent. Hence the unit vector
$$e'''= ||\sigma_i(e'')-\langle \sigma_i(e''), e\rangle \cdot e||^{-1} \cdot (\sigma_i(e'')-\langle \sigma_i(e''), e\rangle \cdot e)$$
 is perpendicular to $e$ and depends continuously on $e$. Thus, we may define the section $s_i$ by setting $s_i(e)=(e, e''')$. 
 
 Next we describe the sections $s_i$ over the sets $F_i$ where $i=k+1, \dots, k+\ell+1$. 
 If $e\in F_i$ then 
 $P_b^{\tau}(e) =0$ and $b=\xi(e)$ lies in $B_{i-k-1}$. In other words, 
 $e\in \dot E(\eta)$ and $b=\xi(e)\in B_{i-k-1}$. The section $\nu_{i-k-1}: B_{i-k-1}\to \dot E(\tau)$ defines a unit vector 
 $\nu_{i-k-1}(b)\in \dot E(\tau)\subset \dot E(\xi)$ which is perpendicular to $e$. Therefore we may define the section 
 $s_i$ of $\ddot \xi$ over $F_i$ by the formula $s_i(e)=(e, \nu_{i-k-1}(e)).$
 
 \end{proof}

%
%
%
%
%

\begin{corollary}
Let $\xi: E\to B$ be a vector bundle such that 
$\xi=\eta\oplus \tau$ where $\tau: E(\tau)\to B$ admits a complex structure and has a nowhere zero continuous section. 
Then
$\TC[\dot\xi : \dot E\to B]\le 2$. 
\end{corollary}
\begin{proof}
This follows from Remark \ref{rk13} and Lemma \ref{lm21}. 
\end{proof}
\begin{corollary}
Let $\xi: E\to B$ be a vector bundle admitting two continuous linearly independent nowhere zero sections. 
Then
$\TC[\dot\xi : \dot E\to B]\le 2$. 
\end{corollary}
\begin{proof}
This reduces to the previous Corollary with $\tau$ the trivial bundle of rank 2. 
\end{proof}

 \bibliographystyle{amsplain}

\end{document}